\documentclass[10pt,a4paper,reqno]{amsart}
\usepackage{amsthm}
\usepackage{amsmath}
\usepackage{amssymb}
\usepackage{graphicx,color}

\usepackage[font=small]{caption}
\usepackage{subcaption}

\usepackage{multirow}
\usepackage[shortlabels]{enumitem}

\makeatletter
\theoremstyle{plain}
\newtheorem{thm}{Theorem}
  \theoremstyle{definition}
  \newtheorem*{thm*}{Theorem}
  
  \theoremstyle{remark}
  
  \theoremstyle{plain}
  \newtheorem{prop}[thm]{Proposition}
  \theoremstyle{plain}
  
  \theoremstyle{plain}
  
 \theoremstyle{definition}
  
  \theoremstyle{remark}
  \newtheorem*{rem*}{Remark}

  \theoremstyle{definition}

\usepackage{amsfonts}
\usepackage{mathrsfs}

\newtheorem*{question*}{\it{QUESTION}}

\theoremstyle{plain}

\newcommand{\N}{\mathbb{N}}
\newcommand{\R}{{\mathbb{R}}}
\newcommand{\C}{{\mathbb{C}}}

\newcommand{\Z}{{\mathbb{Z}}}
\newcommand{\dd}{{\rm d}}
\newcommand{\ii}{{\rm i}}

\newcommand{\Ker}{\mathop\mathrm{Ker}\nolimits}

\def\pPhiq#1#2#3#4#5#6{ % macro for basic hypergeometric functions
  {}_{#1}\phi_{#2}\biggl(\genfrac..{0pt}{}{#3}{#4}\biggl|\,#5;\,#6\biggr)
}

\makeatother

\begin{document}

\title[]{On diagonalizable quantum weighted Hankel matrices}

\author{Franti{\v s}ek {\v S}tampach}
\address[Franti{\v s}ek {\v S}tampach]{
	Department of Mathematics, Faculty of Nuclear Sciences and Physical Engineering, Czech Technical University in Prague, Trojanova~13, 12000 Praha~2, Czech Republic
	}	
\email{stampfra@fjfi.cvut.cz}

\author{Pavel {\v S}{\v t}ov{\' i}{\v c}ek}
\address[Pavel {\v S}{\v t}ov{\' i}{\v c}ek]{
	Department of Mathematics, Faculty of Nuclear Sciences and Physical Engineering, Czech Technical University in Prague, Trojanova~13, 12000 Praha~2, Czech Republic
	}	
\email{stovicek@fjfi.cvut.cz}

\dedicatory{To the memory of Harold Widom (1932-2021).}



\date{\today}

\begin{abstract}
A semi-infinite weighted Hankel matrix with entries defined in terms of basic hypergeometric series is explicitly diagonalized as an operator on $\ell^{2}(\N_{0})$. The approach uses the fact that the operator commutes with a diagonalizable Jacobi operator corresponding to Al-Salam--Chihara orthogonal polynomials. Yet another weighted Hankel matrix, which commutes with a Jacobi operator associated with the continuous $q$-Laguerre polynomials, is diagonalized. As an application, several new integral formulas for selected quantum orthogonal polynomials are deduced. In addition, an open research problem concerning a quantum Hilbert matrix is also mentioned.
\end{abstract}

\maketitle

\section{Introduction}

A great account of research of Harold Widom was devoted to Hankel matrices~\cite{wid_tams66,wid-wil_pams66}. A~prominent Hankel matrix of significant interest is the famous Hilbert matrix
\begin{equation}
 \qquad 
  H_{\nu}:=
 \begingroup % keep the change local
 \setlength\arraycolsep{1pt}\def\arraystretch{1.2} % extended spacing between columns and rows
 \begin{pmatrix}
 \frac{1}{\nu} & \frac{1}{\nu+1} & \frac{1}{\nu+2} & \dots \\
 \frac{1}{\nu+1} & \frac{1}{\nu+2} & \frac{1}{\nu+3} & \dots \\
 \frac{1}{\nu+2} & \frac{1}{\nu+3} & \frac{1}{\nu+4} & \dots \\
 \vdots & \vdots & \vdots & \ddots
 \end{pmatrix}\!, \quad \nu\in\R\setminus(-\N_{0}),
 \endgroup
\label{eq:hilbert_mat}
\end{equation}
which, when regarded as an operator on $\ell^{2}(\N_{0})$, is one of a very few non-trivial examples of Hankel matrices that admit an explicit diagonalization. With the aid of certain previously known identities due to Magnus and Shanker, the diagonalization of $H_{\nu}$ was done by Rosenblum~\cite{ros_pams58}, who found an integral operator whose matrix representation with respect to a suitably chosen orthonormal basis coincides with~$H_{\nu}$ and diagonalizes the integral operator. The passing to the integral representation is not necessary though. Alternatively, the diagonalization of $H_{\nu}$ can be treated by noting that $H_{\nu}$ commutes with a Jacobi operator with an explicitly solvable spectral problem~\cite{kal-sto_lma16}, an approach sometimes referred to as \emph{the commutator method}~\cite{yaf_fap10}.

Let us explain the basic idea of the commutator method in more detail. To a given operator $H$, whose spectral analysis is the ultimate goal, we seek a commuting operator $J$ with simple spectrum and solvable spectral problem. Suppose that $\lambda$ is an eigenvalue of $J$ and $\phi\in\Ker(J-\lambda)$ an eigenvector. Then from equations $J\phi=\lambda\phi$ and $JH=HJ$, one infers that $H\phi\in\Ker(J-\lambda)$. Since the eigenvalue $\lambda$ is simple, there is a number $h=h(\lambda)$ such that $H\phi=h\phi$. If we can assure that $\phi_{0}\neq0$, we may suppose $\phi_{0}=1$, then the eigenvalue $h$ can be computed as follows:
\begin{equation}
h=h\phi_{0}=(H\phi)_{0}=\sum_{n=0}^{\infty}H_{0,n}\phi_{n}.
\label{eq:h_func_basic}
\end{equation}

The Askey scheme of hypergeometric orthogonal polynomials and their $q$-analogues~\cite{koe-les-swa_10} can serve as a~rich source of diagonalizable tridiagonal matrix operators which follows from the well-known relation between spectral properties of Jacobi operators and orthogonal polynomials~\cite{akh_21, ism_09}. A simple tridiagonal structure of the Jacobi matrix is helpful when trying to find a commuting Jacobi operator $J$ in the commutant of~$H$ (i.e., in the space of commuting operators). Moreover, since off-diagonal entries of $J$ are non-vanishing, the spectrum of $J$ is always simple. Of course, the spectrum of $J$ need not be only discrete therefore the basic idea of the commutator method described in the preceding paragraph is to be generalized.

Suppose $J$ is a self-adjoint Jacobi operator determined by the tridiagonal matrix
\begin{equation}
 J= \begin{pmatrix}
 \beta_{0} & \alpha_{0}  \\
 \alpha_{0} & \beta_{1} & \alpha_{1}  \\
 & \alpha_{1} & \beta_{2} & \alpha_{2}  \\
 & & \ddots & \ddots & \ddots
 \end{pmatrix}\!, 
\label{eq:jacobi_mat}
\end{equation}
with $\beta_{n}\in\R$ and $\alpha_{n}\in\R\setminus\{0\}$, and $\{\phi_{n}\}_{n=0}^{\infty}$ the sequence of corresponding orthonormal polynomials defined recursively by equations
\begin{align}
 (\beta_{0}-x)\phi_{0}(x)+\alpha_{0}\phi_{1}(x)&=0,\nonumber\\
 \alpha_{n-1}\phi_{n-1}(x)+(\beta_{n}-x)\phi_{n}(x)+\alpha_{n}\phi_{n+1}(x)&=0, \quad n\geq1, \label{eq:orthonormal_polyn}
\end{align}
and normalization $\phi_{0}(x)=1$. Due to the self-adjointness of $J$, polynomials $\{\phi_{n} \mid n\in\N_{0}\}$ form an orthonormal basis of $L^{2}(\R,\dd\mu)$, where $\mu$ is a unique probability measure on $\R$,
\[
 \int_{\R}\phi_{m}(x)\phi_{n}(x)\dd\mu(x)=\delta_{m,n}, \quad m,n\in\N_{0}.
\]
Moreover, if we denote by $\{e_{n}\mid n\in\N_{0}\}$ the standard basis of $\ell^{2}(\N_{0})$, then the unitary mapping
\[
U:\ell^{2}(\N_{0})\to L^{2}(\R,\dd\mu): e_{n} \mapsto \phi_{n}
\]
diagonalizes $J$, i.e, $UJU^{-1}=T_{\mathrm{id}}$, where $T_{\mathrm{id}}$ is the operator of multiplication by the independent variable acting on $L^{2}(\R,\dd\mu)$. 

Now, if $J$ commutes with a self-adjoint operator $H$, then there exists a measurable function $h$ such that $H=h(J)$; see, for example~\cite[Thm.~1.4, p.~414]{ber-shu_91}. Having the spectral representation of $J$, a determination of~$h$ is the last step of the spectral analysis of $H$. Since $UHU^{-1}=T_{h}$, where $T_{h}$ is the multiplication operator by $h$, function~$h$ can be found using the equation
\begin{equation}
 h=T_{h}1=UHU^{-1}\phi_{0}=UHe_{0}=\sum_{n=0}^{\infty}H_{0,n}\phi_{n},
\label{eq:func_h}
\end{equation}
cf.~\eqref{eq:h_func_basic}. Notice that $1\in L^{2}(\R,\dd\mu)$ since $\mu$ is a probability measure.

The article is organized as follows. In Section~\ref{sec:2}, a brief summary of the current state of the art  of the research focused on diagonalizable weighted Hankel matrices is given. Recall the Askey scheme~\cite{koe-les-swa_10} is divided into two parts: the hypergeometric Askey scheme and their $q$-hypergeometric analogues. First, we explain that the Hilbert matrix is, in a~sense, the only Hankel matrix which can be diagonalized by applying the commutator method to Jacobi operators associated to polynomial families from the hypergeometric Askey scheme. When more degrees of freedom are introduced to the problem by adding non-trivial weights, several explicitly diagonalizable weighted Hankel matrices have already been found. When passing to the $q$-Askey scheme, the applicability of the commutator method is fairly unexplored. As a related and interesting research project, we mention in Section~\ref{sec:2} an open problem concerning diagonalizable quantum analogues of the Hilbert matrix.

Next, in Section~\ref{sec:3}, we initiate a study on diagonalizable weighted Hankel matrices commuting with Jacobi operators associated to polynomial families from the $q$-Askey scheme. First, as the main result of this article, we diagonalize a three-parameter family of weighted Hankel matrices that are found in the commutant of the Jacobi matrix associated to the Al-Salam--Chihara polynomials (Theorem~\ref{thm:spec_H_alsalchih}). We use this result to diagonalize another weighted Hankel matrix corresponding to the continuous $q$-Laguerre polynomials (Theorem~\ref{thm:spec_tildeH_qlag}). As an application, we conclude Section~\ref{sec:3} by deriving several integral formulas for the aforementioned quantum orthogonal polynomials. Finally, selected identities for the $q$-hypergeometric series, which are needed in proofs, are listed in the Appendix for reader's convenience.

\section{The state of the art and an open problem}\label{sec:2}

There are only very few examples of Hankel matrices, regarded as operators on~$\ell^{2}(\N_{0})$, whose spectral problem is solvable explicitly or in terms of standard families of special functions. This contrasts the situation for other well known classes of special operators such as Jacobi, Schr{\" o}dinger, Toeplitz, CMV, etc., where many solvable models exist and  find various applications. This lack of concrete solvable models with Hankel matrices or their weighted generalizations served as a motivation for a  research whose recent achievements are briefly summarized below.

\subsection{The state of the art}

In~\cite{kal-sto_lma16}, the authors observed the three-parameter matrix $B=B(a,b,c)$ with entries
\[
 B_{m,n}=\frac{\Gamma(m+n+a)}{\Gamma(m+n+b+c)}\sqrt{\frac{\Gamma(m+b)\Gamma(m+c)\Gamma(n+b)\Gamma(n+c)}{\Gamma(m+a)\,m!\,\Gamma(n+a)\,n!}}, \quad m,n\in\N_{0},
\]
regarded as an operator on~$\ell^{2}(\N_{0})$, commutes with Jacobi matrix~\eqref{eq:jacobi_mat}, where
\[
 \alpha_{n}=-\sqrt{n(n-1+a)(n-1+b)(n-1+c)} \quad\mbox{ and }\quad 
 \beta_{n}=n(n-1+c)+(n+a)(n+b).
\]
For parameters $a,b,c$ from a suitable domain, this interesting observation yields an explicit diagonalization of $B$ since the commuting Jacobi operator is diagonalizable with the aid of a family of hypergeometric orthogonal polynomials from the Askey scheme called the continuous dual Hahn polynomials. As the entries of $B$ are of the form
\[
 B_{m,n}=w_{m}h_{m+n}w_{n}, \quad m,n\in\N_{0},
\]
for 
\[
 w_{n}=\sqrt{\frac{\Gamma(n+b)\Gamma(n+c)}{\Gamma(n+a)\,n!}}
 \quad\mbox{ and }\quad 
 h_{n}=\frac{\Gamma(m+n+a)}{\Gamma(m+n+b+c)},
\] 
$B$ is the so-called weighted Hankel matrix. In particular, if $a=b$ and $c=1$, $w_{n}=1$ for all $n\in\N_{0}$ and $B$ becomes a Hankel matrix. In fact, $B(\nu,\nu,1)$ coincides with the Hilbert matrix~\eqref{eq:hilbert_mat} and hence the commutator method worked out in detail in~\cite{kal-sto_lma16} provides an alternative way for the diagonalization of the Hilbert matrix.

A natural question is whether there are other Hankel matrices commuting with the diagonalizable Jacobi matrices from the hypergeometric Askey scheme. Unfortunately, the answer is negative. More precisely, it was proven in~\cite{sta-sto_laa20} that, up to an inessential alternating factor, a scalar multiple of the Hilbert matrix is the only Hankel matrix with $\ell^{2}$-columns and rank greater than 1 that can be found in  commutants of Jacobi matrices from the Askey scheme. This fact emphasizes even more the prominent role of the Hilbert matrix. 

On the other hand, if the class of considered Jacobi operators is slightly extended by adding Jacobi operators diagonalizable with the aid of the Stieltjes--Carlitz polynomials~\cite{car_dmj60}, four more diagonalizabe Hankel matrices were found in~\cite[Thm.~6.1]{sta-sto_ieot21} only recently. Stieltjes--Carlitz polynomials do not belong to the hypergeometric Askey scheme since are not given by terminating hypergeometric series. Rather than that, Stieltjes--Carlitz polynomials are intimately related to Jacobian elliptic functions.

When weighted Hankel matrices are considered, several more matrices, to the above mentioned matrix $B$, were successfully diagonalized by applying the commutator method to Jacobi matrices from the Askey scheme. Namely, in~\cite{sta-sto_jmaa19}, four families of weighted Hankel matrices were diagonalized with the aid of Hermite, Laguerre, Meixner, Meixner–Pollaczek, and dual Hahn polynomials.

When passing to the $q$-Askey scheme, i.e., quantum analogues of the classical orthogonal polynomials and corresponding diagonalizable Jacobi operators, the above problems have not been explored yet. In Section~\ref{sec:3}, we initiate the study by diagonalizing two weighted Hankel matrices that commute with Jacobi operators associated to Al-Salam--Chihara and continuous $q$-Laguerre polynomials. Another interesting question is whether a certain quantum analogue to the Hilbert matrix commutes with a tridiagonal matrix and possibly can be diagonalized with the aid of the $q$-Askey scheme. This open problem is partly discussed in the next subsection.

\subsection{An open problem: the quantum Hilbert matrix}

By the quantum Hilbert matrix, one may understand a Hankel matrix with entries dependent on a parameter $q$ which, possibly after a suitable scaling, tend to the entries of the Hilbert matrix as $q\to1$. Such a $q$-analogue of finite order has already appeared in literature. In~\cite{and-ber_jcam09}, the authors derived formulas for the determinant and the inverse of the finite quantum Hilbert matrix whose $(m,n)$-th entry equals 
\[
 \frac{[\nu]_{q}}{[m+n+\nu]_{q}},
\]
where 
\[
 [\alpha]_{q}:=\frac{q^{\alpha/2}-q^{-\alpha/2}}{q^{1/2}-q^{-1/2}}
\]
is the symmetric $q$-deformation of a complex number $\alpha$. Notice that $[\alpha]_{q}\to\alpha$, as $q\to1$.

In a greater generality, a reasonable candidate for the quantum analogue of the Hilbert matrix can be found in the three-parameter family of semi-infinite Hankel matrices $\mathcal{H}_{\nu}=\mathcal{H}_{\nu}(q;\epsilon)$ defined by
\[
 \left(\mathcal{H}_{\nu}\right)_{m,n}:=\frac{q^{\epsilon(m+n)}}{1-q^{m+n+\nu}}, \quad m,n\in\N_{0},
\]
where $q\in(0,1)$, $\nu\in\R\setminus(-\N_{0})$, and $\epsilon>0$. Up to an unimportant multiplicative factor, $\mathcal{H}_{\nu}(q;1/2)$ is a semi-infinite version of the quantum Hilbert matrix from~\cite{and-ber_jcam09}. One can check that $\mathcal{H}_{\nu}$ determines a compact operator on $\ell^{2}(\N_{0})$, for example, by applying Widom's criterion~\cite[Thm.~3.2]{wid_tams66}. In fact, it is not difficult to see that  $\mathcal{H}_{\nu}$ is actually trace class. More concrete results on spectral properties of $\mathcal{H}_{\nu}$ are definitely of interest. The most accessible cases seem to be $\epsilon=1/2$ and $\epsilon=1$.

Further, let us consider the specific case when $\epsilon=\nu=1$ for simplicity and denote $\mathcal{G}:=\mathcal{H}_{1}(q;1/2)$. Hence, we consider the quantum analogue of the classical Hilbert matrix whose entries are the reciprocal quantum integers
\[
 \mathcal{G}_{m,n}=\frac{q^{m+n}}{1-q^{m+n+1}}, \quad m,n\in\N_{0}.
\]
Hoping for a diagonalization of $\mathcal{G}$ possibly in terms of the basic hypergeometric series, one may try to apply the commutator method. Surprisingly, $\mathcal{G}$ commutes with the Jacobi operator $\mathcal{J}$ given by~\eqref{eq:jacobi_mat} and sequences
\[
 \alpha_{n}=-\left(q^{-(n+1)/2}-q^{(n+1)/2}\right)^{\!2} \quad\mbox{ and }\quad \beta_{n}=-4+\left(q^{-1/2}+q^{1/2}\right)\left(q^{-n-1/2}+q^{n+1/2}\right)\!.
\]
Indeed, the commutation relation $\mathcal{GJ}=\mathcal{JG}$ can be straightforwardly verified. Operator $\mathcal{J}$ does not correspond to any polynomial family listed in the $q$-Askey scheme, however. Moreover, to our best knowledge, properties of this operator or the corresponding family of orthogonal polynomials have not been studied yet. Such properties, as for example generating function formulas for the orthogonal polynomials, would be of interest on their own regardless the connection to the quantum Hilbert matrix. 

Without going into details, let us remark that $\mathcal{J}$ determines an unbounded self-adjoint Jacobi operator (i.e. $\mathcal{J}$ restricted to the span of $\{e_{n}\mid n\in\N_{0}\}$ is essentially self-adjoint) which is positive, invertible, and has discrete spectrum. Matrix elements of the inverse read
\[
 \left(\mathcal{J}^{-1}\right)_{m,n}=\sum_{k=\max(m,n)}^{\infty}\frac{1}{\left(q^{-(k+1)/2}-q^{(k+1)/2}\right)^{2}}, \quad m,n\in\N_{0}.
\]
Nevertheless, whether it is possible to analyze spectral properties of $\mathcal{G}$ or $\mathcal{J}$ in a greater detail possibly in terms of commonly known special functions remains an open problem.

\section{Two diagonalizable quantum weighted Hankel matrices}\label{sec:3}

We diagonalize a three-parameter family of weighted Hankel matrices with entries given in terms of the $q$-hypergeometric ${}_{0}\phi_{1}$-function. Recall the definition of the general $q$-hypergeometric series~\cite{gas-rah_04}
\[
 \pPhiq{p}{q}{a_{1},\dots,a_{p}}{b_{1},\dots,b_{q}}{q}{z}:=\sum_{n=0}^{\infty}\frac{(a_{1},\dots,a_{p};q)_{n}}{(b_{1},\dots,b_{q};q)_{n}}(-1)^{(1+q-p)n}\,q^{(1+q-p)n(n-1)/2}\,\frac{z^{n}}{(q;q)_{n}},
\]
where $(a_{1},\dots,a_{p};q)_{n}:=(a_{1};q)_{n}\dots(a_{p};q)_{n}$ and 
\[
 (a;q)_{n}:=\prod_{j=0}^{n-1}\left(1-aq^{j}\right)
\]
is the $q$-Pochhamer symbol. Index $n$ can be taken $\infty$, the convergence of the infinite product is guaranteed by the assumption $|q|<1$. In the notation, we follow the book of Gasper and Rahman~\cite{gas-rah_04}.

The weighted Hankel matrix to be diagonalized is found in the commutant of the Jacobi matrix associated to the Al-Salam--Chihara polynomials. Next, we also diagonalize another weighted Hankel matrix with more explicit entries and commuting with the Jacobi matrix associated to the continuous $q$-Laguerre polynomials. Lastly, as an application, we obtain several integral formulas for the aforementioned orthogonal polynomials that seem to be new.

\subsection{The case of Al-Salam--Carlitz polynomials}

We diagonalize the weighted Hankel matrix $H$ with entries $H_{m,n}=w_{m}h_{m+n}w_{n}$, where the weight reads
\begin{equation}
 w_{n}=\frac{(-a)^{n}q^{n(n-1)/2}}{\sqrt{(q,ab;q)_{n}}}
\label{eq:def_w_alsalchih}
\end{equation}
and the Hankel part is determined by the basic hypergeometric series
\begin{equation}
 h_{n}=\pPhiq{0}{1}{-}{qb/a}{q}{\frac{q^{2-n}}{a^{2}}},
\label{eq:def_h_alsalchih}
\end{equation}
for $n\in\N_{0}$. Hence
\begin{equation}
 H_{m,n}=\frac{(-a)^{n+m}q^{n(n-1)/2+m(m-1)/2}}{\sqrt{(q,ab;q)_{m}(q,ab;q)_{n}}}\,\pPhiq{0}{1}{-}{qb/a}{q}{\frac{q^{2-m-n}}{a^{2}}}, \quad m,n\in\N_{0}.
\label{eq:H_mn_alsalchih}
\end{equation}
If needed, we will write $H=H(a,b)$ to emphasize the dependence on parameters $a$ and $b$ and similarly for $h_{n}=h_{n}(a,b)$ and $w_{n}=w_{n}(a,b)$. The dependence on $q$ is always suppressed in the notation. 
The range for the parameters is restricted to $q\in(0,1)$, $0<|a|<1$, and $|b|<1$.

Let us also introduce the Jacobi operator $J=J(a,b)$ of the form~\eqref{eq:jacobi_mat} with entries determined by sequences
\begin{equation}
 \alpha_{n}=\sqrt{(1-q^{n+1})(1-abq^{n})} \quad\mbox{ and }\quad \beta_{n}=(a+b)q^{n},
\label{eq:def_alp_bet_alsalchih}
\end{equation}
for $n\in\N_{0}$.

\begin{prop}\label{prop:HJ_commute}
 Matrices $H$ and $J$ commute.
\end{prop}

\begin{proof}
Recall that the second Jackson $q$-Bessel function
\[
J_{\nu}(x;q):=\frac{(q^{\nu+1};q)_{\infty}}{(q;q)_{\infty}}\left(\frac{x}{2}\right)^{\nu}\pPhiq{0}{1}{-}{q^{\nu+1}}{q}{-\frac{x^{2}q^{\nu+1}}{4}}
\]
solves the $q$-difference equation~\cite[Eq.~(14.1.23)]{ism_09}
\[
 J_{\nu}(q^{-1/2}x;q)-\left(q^{-\nu/2}+q^{\nu/2}\right)J_{\nu}(x;q)+\left(1+\frac{x^{2}}{4}\right)J_{\nu}(q^{1/2}x;q)=0.
\]
It follows that sequence~\eqref{eq:def_h_alsalchih} satisfies the recurrence
\[
 (ab-q^{1-k})h_{k-1}-a(a+b)h_{k}+a^{2} h_{k+1}=0,
\]
for $k\in\Z$. Next, by writing $k=m+n$ in the above equation, one verifies the identity
\[
(\beta_{m}-\beta_{n})H_{m,n}+\alpha_{m-1}H_{m-1,n}+\alpha_{m}H_{m+1,n}-\alpha_{n-1}H_{m,n-1}-\alpha_{n}H_{m,n+1}=0,
\]
for matrix entries $H_{m,n}=w_{m}h_{m+n}w_{n}$, where $w_{n}$ is as in~\eqref{eq:def_w_alsalchih} and 
$\alpha_{n}$ and $\beta_{n}$ given by~\eqref{eq:def_alp_bet_alsalchih}; by convention, we also put $H_{-1,n}=H_{n,-1}:=0$ for any $n\in\N_{0}$. This means nothing but the matrix equality $JH-HJ=0$.
\end{proof}

Orthogonal polynomials determined by Jacobi parameters~\eqref{eq:def_alp_bet_alsalchih} are the Al-Salam--Chihara polynomials $Q_{n}(x;a,b\,|\, q)$ since they are given by the recurrence~\cite[Eq.~(14.8.4)]{koe-les-swa_10}
\[
(1-q^{n})(1-abq^{n-1})Q_{n-1}(x;a,b\,|\, q)+\left((a+b)q^{n}-2x\right)Q_{n}(x;a,b\,|\, q)+Q_{n+1}(x;a,b \,|\,q)=0
\]
and $Q_{-1}(x;a,b\,|\, q)=0$, $Q_{0}(x;a,b\,|\, q)=1$.  The corresponding orthonormal polynomials fulfill
\begin{equation}
 \phi_{n}(2x)=\frac{1}{\sqrt{(q,ab;q)_{n}}}\,Q_{n}(x;a,b\,|\, q), \quad n\in\N_{0},
\label{eq:phi_alsalchih}
\end{equation}
and form an orthonormal basis in the Hilbert space $L^{2}((-1,1),\dd\mu)$, where $\mu$ is the~absolutely continuous orthogonality measure determined by the density
\begin{equation}
\frac{\dd\mu}{\dd x}(\cos\theta)=\frac{(q,ab;q)_{\infty}}{2\pi\sin\theta}\left|\frac{(e^{2\ii\theta};q)_{\infty}}{(ae^{\ii\theta},be^{\ii\theta};q)_{\infty}}\right|^{2},
\label{eq:mu_meas_alsalchih}
\end{equation}
for $x=\cos\theta$ and $\theta\in(0,\pi)$; see~\cite[\S~14.8]{koe-les-swa_10}.

Thus, by the commutator method, there exists a Borel function~$h$ on $(-1,1)$ such that $UHU^{-1}=T_{h}$, where the unitary mapping $U:\ell^{2}(\N_{0})\to L^{2}((-1,1),\dd\mu)$ is determined by the correspondence $U:e_{n}\mapsto\phi_{n}(2\,\cdot)$, $n\in\N_{0}$. Function $h$ satisfies
\begin{equation}
h(x)=h(x)\phi_{0}(2x)=\sum_{n=0}^{\infty}H_{0,n}\phi_{n}(2x)=w_{0}\sum_{n=0}^{\infty}h_{n}w_{n}\phi_{n}(2x),
\label{eq:func_h_alsalchih}
\end{equation}
which means that 
\begin{equation}
 h(x)=\sum_{n=0}^{\infty}\frac{(-a)^{n}q^{n(n-1)/2}}{(q,ab;q)_{n}}\,\pPhiq{0}{1}{-}{qb/a}{q}{\frac{q^{2-n}}{a^{2}}}Q_{n}(x;a,b\,|\, q),
\label{eq:h_gen_func}
\end{equation}
where we have substituted from~\eqref{eq:def_w_alsalchih}, \eqref{eq:def_h_alsalchih}, and~\eqref{eq:phi_alsalchih}. At this point, the diagonalization of $H$ is a matter of a possible simplification of the expression on the right-hand side in~\eqref{eq:h_gen_func}, which miraculously simplifies, indeed.

\begin{prop}\label{prop:func_h_alsalchih}
For $\theta\in(0,\pi)$, we have
\begin{equation}
 h(\cos\theta)=\frac{(ae^{-\ii\theta}, ae^{\ii\theta},qe^{-\ii\theta}/a,qe^{\ii\theta}/a;q)_{\infty}}{(ab,qb/a;q)_{\infty}}.
\label{eq:h_alsalchih}
\end{equation}
\end{prop}

\begin{proof}
The starting point is the generating function formula for the Al-Salam--Carlitz polynomials~\cite[Eq.~(14.8.16)]{koe-les-swa_10}
 \[
  \sum_{n=0}^{\infty}\frac{(\gamma;q)_{n}\,t^{n}}{(q;ab;q)_{n}}\,Q_{n}(x;a,b\,|\,q)=\frac{(\gamma e^{\ii\theta}t;q)_{\infty}}{(e^{\ii\theta}t;q)_{\infty}}\,\pPhiq{3}{2}{\gamma, ae^{\ii\theta}, be^{\ii\theta}}{ab,\gamma e^{\ii\theta}t}{q}{e^{-\ii\theta}t},
 \]
 where $|t|<1$ and $\gamma\in\C$. Here and everywhere below, $x=\cos\theta$ with $\theta\in(0,\pi)$ fixed. By putting $t=z/\gamma$ and sending $\gamma\to\infty$ in the above formula, we obtain 
 \[
  \sum_{n=0}^{\infty}\frac{q^{n(n-1)/2}(-z)^{n}}{(q;ab;q)_{n}}\,Q_{n}(x;a,b\,|\,q)=(z e^{\ii\theta};q)_{\infty}\,\pPhiq{2}{2}{ae^{\ii\theta}, be^{\ii\theta}}{ab,ze^{\ii\theta}}{q}{ze^{-\ii\theta}},
 \]
 for $z\in\C$. Next, by setting $z=aq^{-m}$, multiplying both sides by $q^{m(m+1)}a^{-2m}/(q;qb/a;q)_{m}$, and summing up for $m=0,1,\dots$, we deduce from~\eqref{eq:h_gen_func} the formula
 \[
  h(x)=\sum_{m=0}^{\infty}\frac{(aq^{-m}e^{\ii\theta};q)_{\infty}}{(q;qb/a;q)_{m}}\frac{q^{m(m+1)}}{a^{2m}}\,\pPhiq{2}{2}{ae^{\ii\theta}, be^{\ii\theta}}{ab,aq^{-m}e^{\ii\theta}}{q}{aq^{-m}e^{-\ii\theta}}.
 \]
Further, we apply formula~\eqref{eq:q-pfaff-kummer} with
\[
 a\leftarrow ae^{\ii\theta}, \quad b\leftarrow \frac{a}{b}q^{-m}, \quad c\leftarrow aq^{-m}e^{\ii\theta}, \quad z\leftarrow be^{-\ii\theta},
\]
which yields
\begin{equation}
 h(x)=\frac{(be^{-\ii\theta};q)_{\infty}}{(ab;q)_{\infty}}\sum_{m=0}^{\infty}\frac{(aq^{-m}e^{\ii\theta};q)_{\infty}}{(q;qb/a;q)_{m}}\frac{q^{m(m+1)}}{a^{2m}}\,\pPhiq{2}{1}{ae^{\ii\theta}, aq^{-m}/b}{aq^{-m}e^{\ii\theta}}{q}{be^{-\ii\theta}}.
\label{eq:h_qid1_inproof}
\end{equation}

As the next step, we apply identity~\eqref{eq:q-three-term} to the ${}_{2}\phi_{1}$-function in~\eqref{eq:h_qid1_inproof}. Moreover, the coefficients given in terms of $q$-Pochhammer symbols slightly simplify with the aid of the identity
\[
 (\alpha q^{-m},q^{m+1}/\alpha;q)_{\infty}=(-\alpha)^{m}q^{-m(m+1)/2}(\alpha,q/\alpha;q)_{\infty},
\]
which holds true for all $\alpha\in\C\setminus\{0\}$ and $m\in\N_{0}$. The resulting expression reads
\begin{align}
 h(x)&=\frac{(be^{-\ii\theta},ae^{-\ii\theta}, ae^{\ii\theta},qe^{-\ii\theta}/a;q)_{\infty}}{(ab,qb/a,e^{-2\ii\theta};q)_{\infty}}\sum_{m=0}^{\infty}\frac{q^{m(m+1)/2}}{(q;q)_{m}}\left(-\frac{e^{\ii\theta}}{a}\right)^{\!m}\pPhiq{2}{1}{be^{\ii\theta}, qe^{\ii\theta}/a}{qe^{2\ii\theta}}{q}{q^{m+1}}\nonumber\\
 &+\mbox{c.c.},
\label{eq:h_qid2_inproof}
\end{align}
where the abbreviation c.c. stands for the term which equals the complex conjugation
to the previous term.

Using the definition of the ${}_{2}\phi_{1}$-function in~\eqref{eq:h_qid2_inproof} and interchanging the order of summation, we observe that the first term in~\eqref{eq:h_qid2_inproof}, up to the multiplicative factor, is equal to
\[
\sum_{n=0}^{\infty}\frac{(be^{\ii\theta}, qe^{\ii\theta}/a;q)_{n}}{(q,qe^{2\ii\theta};q)_{n}}q^{n}\,\pPhiq{0}{0}{-}{-}{q}{\frac{e^{\ii\theta}q^{n+1}}{a}}=(qe^{\ii\theta}/a;q)_{\infty}\,\pPhiq{2}{1}{be^{\ii\theta},0}{qe^{2\ii\theta}}{q}{q},
\]
where we have used~\eqref{eq:q-exp}. Hence we have
\[
 h(x)=\frac{(ae^{-\ii\theta}, ae^{\ii\theta},qe^{-\ii\theta}/a,qe^{\ii\theta}/a;q)_{\infty}}{(ab,qb/a;q)_{\infty}}
 \left[
 \frac{(be^{-\ii\theta};q)_{\infty}}{(e^{-2\ii\theta};q)_{\infty}}\,\pPhiq{2}{1}{be^{\ii\theta},0}{qe^{2\ii\theta}}{q}{q}+\mbox{c.c.}\,
 \right]\!.
\]
Finally, it suffices to notice that by~\eqref{eq:q-vandermonde}, the expression in the square brackets equals~1.
\end{proof}

In total, we have deduced a full spectral representation of $H$ which is summarized in the next theorem.

\begin{thm}\label{thm:spec_H_alsalchih}
For $a,b\in\R$ such that $0<|a|<1$ and $|b|<1$, operator $H$ with matrix entries~\eqref{eq:H_mn_alsalchih} is unitarily equivalent to the operator of multiplication by the function 
\[
 h(x)=\frac{(ae^{-\ii\theta}, ae^{\ii\theta},qe^{-\ii\theta}/a,qe^{\ii\theta}/a;q)_{\infty}}{(ab,qb/a;q)_{\infty}}, \quad x=\cos\theta,
\]
acting on $L^{2}((-1,1),\dd\mu)$, where measure $\mu$ is given by~\eqref{eq:mu_meas_alsalchih}. In particular, 
the spectrum of~$H$ is simple, purely absolutely continuous, and fills the interval
\[
 \sigma_{\mathrm{ac}}(H)=\frac{1}{(ab,qb/a;q)_{\infty}}\left[(|a|,q/|a|;q)_{\infty}^{2},(-|a|,-q/|a|;q)_{\infty}^{2}\right].
\]
Consequently, the operator norm of $H$ reads
\[
\|H\|=\frac{(-|a|,-q/|a|;q)_{\infty}^{2}}{|(ab,qb/a;q)_{\infty}|}.
\]
\end{thm}

\subsection{The case of continuous $q$-Laguerre polynomials}

With the aid of results of the previous subsection, we diagonalize the weighted Hankel matrix 
\begin{equation}
 \tilde{H}_{m,n}=\tilde{H}_{m,n}(\alpha;q):=\frac{q^{(m-n)^{2}/2}\,(q^{\alpha+1};q){}_{m+n}}{\sqrt{(q^{2},q^{2\alpha+2};q^{2})_{m}(q^{2},q^{2\alpha+2};q^{2})_{n}}},\quad m,n\in\N_{0},
\label{eq:tildeH_mn_qlag}
\end{equation}
where $q\in(0,1)$ and $\alpha>-1$. In the course of the diagonalization, we will work with the closely related matrix
\begin{equation}
 G_{m,n}=G_{m,n}(a;q):=\frac{q^{(m-n)^{2}/4}\,(aq^{1/4};q^{1/2}){}_{m+n}}{\sqrt{(q,a^{2}q^{1/2};q)_{m}(q,a^{2}q^{1/2};q)_{n}}},\quad m,n\in\N_{0},
\label{eq:Gmn_qlag}
\end{equation}
rather than $\tilde{H}$. Notice that $G(q^{\alpha+1/2};q^{2})=\tilde{H}(\alpha;q)$. The relation between matrices $G$ and~$H$ from~\eqref{eq:H_mn_alsalchih} reveals the following statement.

\begin{prop}\label{prop:lin_comb_HG}
 One has
 \begin{equation}
  G(a;q)=AH(a,aq^{1/2})+BH(aq^{1/2},a),
 \label{eq:lin_comb_HG}
 \end{equation}
 where
 \[
  A=-\frac{q^{1/4}}{a(1-q^{1/2})(q^{1/4}/a;q^{1/2})_{\infty}} \quad\mbox{ and }\quad B=\frac{1}{(q^{1/4}/a;q^{1/2})_{\infty}}.
 \]
 Consequently, matrices $G(a;q)$ and $J(a,aq^{1/2})$ commute.
\end{prop}

\begin{proof}
 Equation~\eqref{eq:lin_comb_HG} means that
 \[
 AH_{m,n}(a,aq^{1/2})+BH_{m,n}(aq^{1/2},a)=G_{m,n}(a;q), \quad \forall m,n\in\N_{0},
 \]
 which, when we use~\eqref{eq:def_w_alsalchih} and~\eqref{eq:Gmn_qlag}, gets the form
 \[
  A\left(-aq^{-1/2}\right)^{k}h_{k}(a,aq^{1/2}) + B\left(-a\right)^{k}h_{k}(aq^{1/2},a)=q^{-k^{2}/4}(q^{1/4}a;q^{1/2})_{k},
 \]
 for $k:=m+n$. Using also~\eqref{eq:def_h_alsalchih}, we see that the claim holds if the identity
 \begin{equation}
 A\left(-aq^{-1/2}\right)^{k}\pPhiq{0}{1}{-}{q^{3/2}}{q}{\frac{q^{2-k}}{a^{2}}}+
 B\left(-a\right)^{k}\pPhiq{0}{1}{-}{q^{1/2}}{q}{\frac{q^{1-k}}{a^{2}}}=q^{-k^{2}/4}(q^{1/4}a;q^{1/2})_{k}
 \label{eq:lin_comb_inproof}
 \end{equation}
 is true for all $k\in\N_{0}$.
 
 It is straightforward to decompose the $q$-exponential~\eqref{eq:q-exp} into the sum of its odd and even part. It results in the identity
 \[
 -\frac{z}{1-q}\,\pPhiq{0}{1}{-}{q^{3}}{q^{2}}{q^{3}z^{2}}+\pPhiq{0}{1}{-}{q}{q^{2}}{qz^{2}}=(z;q)_{\infty},
 \]
 for $z\in\C$. By substituting $z=q^{-k+1/2}/a$ and using that 
 \[
q^{-k^{2}/2}\,(aq^{1/2};q){}_{k}=\frac{(q^{-k+1/2}/a;q){}_{\infty}}{(q^{1/2}/a;q){}_{\infty}}\,(-a)^{k},
 \]
 for any $k\in\N_{0}$, we arrive at identity~\eqref{eq:lin_comb_inproof} with $q$ replaced by $q^{2}$. This proves the first claim.
 
 To verify the second claim, it suffices to note that $J(a,b)$ commute with $H(a,b)$ as well as with $H(b,a)$ which follows from Proposition~\ref{prop:HJ_commute} and the symmetry $J(a,b)=J(b,a)$, see~\eqref{eq:def_alp_bet_alsalchih}. Then it follows from~\eqref{eq:lin_comb_HG} that $G(a;q)$ and $J(a,aq^{1/2})$ commute.
\end{proof}

Jacobi matrix $J(a,aq^{1/2})$, with $a=q^{\alpha/2+1/4}$, corresponds to the continuous $q$-Laguerre polynomials $P_{n}^{(\alpha)}(\cdot\mid q)$ that are a special case of the Al-Salam--Chihara polynomials, see~\cite[\S~14.19]{koe-les-swa_10}. More precisely, one has
\[
 P_{n}^{(\alpha)}(x\mid q)=\frac{a^{n}}{(q;q)_{n}}Q_{n}(x;a,aq^{1/2}\mid q),
\]
for $a=q^{\alpha/2+1/4}$. Hence, by~\eqref{eq:phi_alsalchih} and~\eqref{eq:mu_meas_alsalchih}, functions
\begin{equation}
 \phi_{n}(2x)=\frac{1}{\sqrt{(q,a^{2}q^{1/2};q)_{n}}}\,Q_{n}(x;a,a q^{1/2}\,|\, q)=\sqrt{\frac{(q;q)_{n}}{(a^{2}q^{1/2};q)_{n}}}\,a^{-n}P_{n}^{(\alpha)}(x\mid q), \quad n\in\N_{0},
\label{eq:phi_qlag}
\end{equation}
form an orthonormal basis in the Hilbert space $L^{2}((-1,1),\dd\mu)$, where 
\begin{equation}
\frac{\dd\mu}{\dd x}(\cos\theta)=
\frac{(q,a^{2}q^{1/2};q)_{\infty}}{2\pi\sin\theta}\left|\frac{(e^{2i\theta};q)_{\infty}}{(ae^{\ii\theta};q^{1/2})_{\infty}}\right|^{2},\quad x=\cos\theta,
\label{eq:mu_meas_qlag}
\end{equation}
and  $a=q^{\alpha/2+1/4}$ (here, we do not designate the dependence on $a$ and $b=aq^{1/2}$ in the notation of~$\phi$ and $\mu$).

Analogously to the case of $H$, the unitary mapping $U:\ell^{2}(\N_{0})\to L^{2}((-1,1),\dd\mu):e_{n}\mapsto\phi_{n}(2\,\cdot)$ diagonalizes $G$, i.e., $UGU^{-1}=T_{g}$, where 
\begin{equation}
 g(x)=\sum_{n=0}^{\infty}G_{0,n}\phi_{n}(2x).
\label{eq:func_g_qlag}
\end{equation}

\begin{prop}
 For $\theta\in(0,\pi)$, we have
 \[
  g(\cos\theta)=\frac{(q^{1/2};q)_{\infty}(-q^{1/4}e^{\ii\theta},-q^{1/4}e^{-\ii\theta};q^{1/2})_{\infty}}{(-aq^{1/4};q^{1/2})_{\infty}}.
 \]
\end{prop}

\begin{proof}
 It follows from~\eqref{eq:func_g_qlag} and Proposition~\ref{prop:lin_comb_HG} that
 \[
  g(x)=A\sum_{n=0}^{\infty}H_{n,0}(a,aq^{1/2})\phi_{n}(2x)+B\sum_{n=0}^{\infty}H_{n,0}(aq^{1/2},a)\phi_{n}(2x),
 \]
 which, when compared to~\eqref{eq:func_h_alsalchih}, yields
 \[
  g(x)=Ah(x;a,aq^{1/2})+Bh(x;aq^{1/2},a)
 \]
 where we have designated the dependence on parameters $a$, $b$ in notation $h(x)=h(x;a,b)$ for obvious reasons. Using Proposition~\ref{prop:func_h_alsalchih}, we obtain
 \begin{align*}
  g(\cos\theta)=\frac{1}{(a^{2}q^{1/2},q^{1/2};q)_{\infty}(q^{1/4}/a;q^{1/2})_{\infty}}&\bigg[(aq^{1/2}e^{\ii\theta},aq^{1/2}e^{-\ii\theta},q^{1/2}e^{\ii\theta}/a,q^{1/2}e^{-\ii\theta}/a;q)_{\infty}\\
  &\hskip34pt-\frac{q^{1/4}}{a}(ae^{\ii\theta},ae^{-\ii\theta},qe^{\ii\theta}/a,qe^{-\ii\theta}/a;q)_{\infty}\bigg],
 \end{align*}
 for $\theta\in(0,\pi)$. Finally, applying identity~\eqref{eq:ex_2.16_gas-rah} to the expression in the square brackets, we arrive at the formula from the statement.
\end{proof}

Recalling that $G(q^{\alpha+1/2};q^{2})=\tilde{H}(\alpha;q)$, we may summarize the obtained results on the diagonalization of~$\tilde{H}$ as follows.

\begin{thm}\label{thm:spec_tildeH_qlag}
For $\alpha>-1$, operator $\tilde{H}$ with matrix entries~\eqref{eq:tildeH_mn_qlag} is unitarily equivalent to the operator of multiplication by the function 
\begin{equation}
 \tilde{h}(x)=\frac{(q;q^{2})_{\infty}(q^{1/2}e^{\ii\theta}, q^{1/2}e^{-\ii\theta};q)_{\infty}}{(-q^{\alpha+1};q)_{\infty}}, \quad x=\cos\theta,
\label{eq:tildeh_qlag}
\end{equation}
acting on $L^{2}((-1,1),\dd\mu)$, where $\mu$ is given by~\eqref{eq:mu_meas_qlag} with $q$ replaced by $q^{2}$. In particular, 
the spectrum of $\tilde{H}$ is simple, purely absolutely continuous, and fills the interval
\[
 \sigma_{\mathrm{ac}}(\tilde{H})=\frac{(q;q^{2})_{\infty}}{(-q^{\alpha+1};q)_{\infty}}\left[(q^{1/2};q)_{\infty}^{2},(-q^{1/2};q)_{\infty}^{2}\right].
\]
Consequently, the operator norm of $\tilde{H}$ reads
\[
\|\tilde{H}\|=\frac{(q;q^{2})_{\infty}(-q^{1/2};q)_{\infty}^{2}}{(-q^{\alpha+1};q)_{\infty}}.
\]
\end{thm}

\subsection{Application: Integral formulas for quantum orthogonal polynomials}

Recall that, in Theorem~\ref{thm:spec_H_alsalchih}, we have diagonalized $H$ in the sense that 
\[
UHU^{-1}=T_{h},
\]
where $T_{h}$ is the operator of multiplication by $h$ acting on $L^{2}((-1,1),\dd\mu)$ and the unitary mapping $U:\ell^{2}(\N_{0})\to L^{2}((-1,1),\dd\mu)$ is unambiguously determined by the correspondence $Ue_{n}=\phi_{n}(2\,\cdot)$ for all $n\in\N_{0}$. Measure $\mu$ is given by density~\eqref{eq:mu_meas_alsalchih} and polynomials $\phi_{n}$ by~\eqref{eq:phi_alsalchih}. It follows that
\[
 H_{m,n}=\langle e_{m}, He_{n}\rangle_{\ell^{2}(\N_{0})}=\langle \phi_{m}(2\,\cdot), T_{h}\phi_{n}(2\,\cdot)\rangle_{L^{2}((-1,1),\dd\mu)}=\int_{-1}^{1}h(x)\phi_{m}(2x)\phi_{n}(2x)\dd\mu(x),
\]
for all $m,n\in\N_{0}$. Substituting for $x=\cos\theta$ in the integral and using formulas~\eqref{eq:H_mn_alsalchih}, \eqref{eq:phi_alsalchih}, \eqref{eq:mu_meas_alsalchih}, and~\eqref{eq:h_alsalchih}, we obtain the following non-trivial integral identity for Al-Salam--Chihara polynomials:
\begin{align}
\frac{(q;q)_{\infty}}{2\pi(qb/a;q)_{\infty}} &\int_{0}^{\pi}Q_{m}(\cos\theta;a,b\mid q)\,Q_{n}(\cos\theta;a,b\mid q)\left|\frac{(e^{2\ii\theta},q e^{\ii\theta}/a;q)_{\infty}}{(be^{\ii\theta};q)_{\infty}}\right|^{2}\dd\theta\nonumber\\
&\hskip70pt=(-a)^{n+m}\,q^{\frac{m(m-1)+n(n-1)}{2}}\,\pPhiq{0}{1}{-}{qb/a}{q}{\frac{q^{2-m-n}}{a^{2}}},
\label{eq:int_id_alsalchih}
\end{align}
which holds true for all $m,n\in\N_{0}$, $q\in(0,1)$, and $a,b\in\R$ such that $0<|a|<1$ and $|b|<1$.

Analogously, using formulas~\eqref{eq:tildeH_mn_qlag}, \eqref{eq:phi_qlag}, \eqref{eq:mu_meas_qlag}, and~\eqref{eq:tildeh_qlag} obtained in the course of the diagonalization of~$\tilde{H}$, one deduces the integral formula for the continuous $q$-Laguerre polynomials:
\begin{align}
\frac{(q;q^{\alpha+1};q)_{\infty}}{2\pi} &\int_{0}^{\pi}P_{m}^{(\alpha)}(\cos\theta\mid q^{2})\,P_{n}^{(\alpha)}(\cos\theta\mid q^{2})\left|\frac{(e^{\ii\theta},-e^{\ii\theta},q^{1/2}e^{\ii\theta};q)_{\infty}}{(q^{\alpha+1/2}e^{\ii\theta};q)_{\infty}}\right|^{2}\dd\theta\nonumber\\
&\hskip108pt=\frac{q^{(\alpha+1/2)(m+n)+(m-n)^{2}/2}\,(q^{\alpha+1};q)_{m+n}}{(q^{2};q^{2})_{m}(q^{2};q^{2})_{n}},
\label{eq:int_id_qlag}
\end{align}
for all $m,n\in\N_{0}$, $q\in(0,1)$, and $\alpha>-1$. In fact, there is another $q$-analogue to the Laguerre polynomials, see~\cite[Eq.~(14.19.17)]{koe-les-swa_10}, that are related to the continuous $q$-Laguerre polynomials by the quadratic transformation 
\[
P_{n}^{(\alpha)}(x;q)=q^{-\alpha n}P_{n}^{(\alpha)}(x\mid q^{2}).
\]
Identity~\eqref{eq:int_id_qlag} written in terms of polynomials $P_{n}^{(\alpha)}(\,\cdot\,;q)$ becomes
\begin{align*}
\frac{(q;q^{\alpha+1};q)_{\infty}}{2\pi} &\int_{0}^{\pi}P_{m}^{(\alpha)}(\cos\theta;q)\,P_{n}^{(\alpha)}(\cos\theta; q)\left|\frac{(e^{\ii\theta},-e^{\ii\theta},q^{1/2}e^{\ii\theta};q)_{\infty}}{(q^{\alpha+1/2}e^{\ii\theta};q)_{\infty}}\right|^{2}\dd\theta\nonumber\\
&\hskip148pt=\frac{q^{(m+n)/2+(m-n)^{2}/2}\,(q^{\alpha+1};q)_{m+n}}{(q^{2};q^{2})_{m}(q^{2};q^{2})_{n}},
\end{align*}
where $m,n\in\N_{0}$, $q\in(0,1)$, and $\alpha>-1$.

Yet another similar identity can be deduced for the continuous $q$-Laguerre polynomials directly from~\eqref{eq:int_id_alsalchih} by using the equation
\[
 Q_{n}(x;q^{\frac{\alpha}{2}+\frac{1}{4}},q^{\frac{\alpha}{2}+\frac{3}{4}}\mid q)=\frac{(q;q)_{n}}{q^{\left(\frac{\alpha}{2}+\frac{1}{4}\right)n}}P_{n}^{(\alpha)}(x\mid q),
\]
see the first limit relation in~\cite[\S~14.19]{koe-les-swa_10}. Thus, putting $a=q^{\frac{\alpha}{2}+\frac{1}{4}}$ and $b=q^{\frac{\alpha}{2}+\frac{3}{4}}$ in~\eqref{eq:int_id_alsalchih}, one obtains
\begin{align*}
\frac{(q;q)_{\infty}}{2\pi(q^{3/2};q)_{\infty}} &\int_{0}^{\pi}P_{m}^{(\alpha)}(\cos\theta\mid q)\,P_{n}^{(\alpha)}(\cos\theta\mid q)\left|\frac{(e^{2\ii\theta},q^{\frac{3}{4}-\frac{\alpha}{2}}e^{\ii\theta};q)_{\infty}}{(q^{\frac{3}{4}+\frac{\alpha}{2}}e^{\ii\theta};q)_{\infty}}\right|^{2}\dd\theta\nonumber\\
&\hskip2pt=(-1)^{m+n}\,\frac{q^{(\alpha+1/2)(m+n)+m(m-1)/2+n(n-1)/2}}{(q;q)_{m}(q;q)_{n}}\,\pPhiq{0}{1}{-}{q^{3/2}}{q}{q^{-m-n-\alpha+3/2}},
\end{align*}
for $m,n\in\N_{0}$, $q\in(0,1)$, and $\alpha>-1$.

We also mention another special case of~\eqref{eq:int_id_alsalchih} related to a $q$-analogue of Hermite polynomials. If $b=0$, the Al-Salam--Chihara polynomials becomes the continuous big $q$-Hermite polynomials
\[
 H_{n}(x;a\mid q)=Q_{n}(x;a,0\mid q),
\]
see~\cite[\S~14.18]{koe-les-swa_10}. In this particular case, we have
\begin{align*}
\frac{(q;q)_{\infty}}{2\pi} &\int_{0}^{\pi}H_{m}(\cos\theta;a\mid q)\,H_{n}(\cos\theta;a\mid q)\left|(e^{2\ii\theta},q e^{\ii\theta}/a;q)_{\infty}\right|^{2}\dd\theta\\
&\hskip132pt=(-a)^{n+m}\,q^{\frac{m(m-1)+n(n-1)}{2}}\,\pPhiq{0}{1}{-}{0}{q}{\frac{q^{2-m-n}}{a^{2}}},
\end{align*}
where $m,n\in\N_{0}$, $q\in(0,1)$, and $0<|a|<1$.

\section*{Acknowledgement}
The research of F.~{\v S}. was supported by grant No.~20-17749X of the Czech Science Foundation. P.~{\v S}. acknowledges partial support by the European Regional Development Fund-Project “Center for Advanced Applied Science” No. CZ.02.1.01/0.0/0.0/16\_019/0000778.

\section*{Appendix: Selected formulas for basic hypergeometric series}

For reader's convenience, we list 5 selected identities for basic hypergeometric series and admissible parameters that are used in the proofs. All of them are borrowed directly from~\cite{gas-rah_04}.

One of the $q$-exponential functions is~\cite[Eq.~(II.2)]{gas-rah_04}
\begin{equation}
 \pPhiq{0}{0}{-}{0}{q}{z}=(z;q)_{\infty}.
\label{eq:q-exp}
\end{equation}
Jackson's $q$-analogue of the Pfaff--Kummer formula reads~\cite[Eq.~(1.5.4)]{gas-rah_04}
\begin{equation}
\pPhiq{2}{1}{a,b}{c}{q}{z}=\frac{(az;q)_{\infty}}{(z;q)_{\infty}}\,\pPhiq{2}{2}{a,c/b}{c,az}{q}{bz}.
\label{eq:q-pfaff-kummer}
\end{equation}
Three term transformation~\cite[Eq.~(III.31)]{gas-rah_04} together with Heine's transformation formula~\cite[Eq.~(1.4.1)]{gas-rah_04} yields the identity
\begin{align}
\pPhiq{2}{1}{a,b}{c}{q}{z}&=\frac{(abz/c,q/c;q)_{\infty}}{(az/c,q/a;q)_{\infty}}\,\pPhiq{2}{1}{c/a,cq/abz}{cq/az}{q}{bq/c}\nonumber\\
&\hskip8pt-\frac{q}{az}\frac{(b,c/a,az/q,q^{2}/az;q)_{\infty}}{(c,q/a,c/az,z;q)_{\infty}}\,\pPhiq{2}{1}{q/b,z}{aqz/c}{q}{bq/c}.
\label{eq:q-three-term}
\end{align}
The particular case of the non-terminating $q$-Vandermonde identity with $b=0$ reads
\begin{equation}
\frac{(aq/c;q)_{\infty}}{(q/c;q)_{\infty}}\,\pPhiq{2}{1}{a,0}{c}{q}{q}
+\frac{(a;q)_{\infty}}{(c/q;q)_{\infty}}\,\pPhiq{2}{1}{aq/c,0}{q^{2}/c}{q}{q}=1,
\label{eq:q-vandermonde}
\end{equation}
see \cite[Eq.~(II.23)]{gas-rah_04}. Finally, the particular case of the identity from~\cite[Ex.~2.16(ii)]{gas-rah_04} with $\mu=e^{\ii\theta}$ and $\lambda=a$ yields the equality
\begin{align}
  &(aq^{1/2}e^{\ii\theta},aq^{1/2}e^{-\ii\theta},q^{1/2}e^{\ii\theta}/a,q^{1/2}e^{-\ii\theta}/a;q)_{\infty}
  -\frac{q^{1/4}}{a}(ae^{\ii\theta},ae^{-\ii\theta},qe^{\ii\theta}/a,qe^{-\ii\theta}/a;q)_{\infty}\nonumber\\
  &=(q^{1/2},q^{1/2},aq^{1/4},aq^{3/4},q^{1/4}/a,q^{3/4}/a,-q^{1/4}e^{\ii\theta},-q^{3/4}e^{\ii\theta},-q^{1/4}e^{-\ii\theta},-q^{3/4}e^{-\ii\theta};q)_{\infty}.\nonumber\\
%  &=(q^{1/2};q)_{\infty}^{2}(aq^{1/4},q^{1/4}/a,-q^{1/4}e^{\ii\theta},-q^{1/4}e^{-\ii\theta};q^{1/2})_{\infty}
  \label{eq:ex_2.16_gas-rah}
\end{align}

\bibliographystyle{acm}

\end{document}